\documentclass[letterpaper, 10 pt, conference]{ieeeconf}  

\IEEEoverridecommandlockouts                              

\usepackage{amsmath, bm, amsfonts, amssymb, bbold, dsfont, graphicx, mathtools, mathrsfs, cancel, float} 
\newtheorem{claim}{Claim}[section]
\newtheorem{lemma}{Lemma}[section]

\DeclareMathOperator*{\argmin}{arg\,min}

\def\R{\mathds{R}}

\def\F{\mathcal{F}}
\def\Ftilde{\Tilde{\mathcal{F}}}
\def\P{\mathcal{P}}
\def\K{\mathcal{K}}

\def\M{\mathcal{M}}
\def\I{\mathcal{I}}

\def\S{\bm S}
\def\St{\bm S^t}
\def\x{\bm x}
\def\y{\bm y}
\def\g{\bm g}

\def\X{\bm X}

\def\data{\bm \Psi(\bm X)}
\def\dataplus{\bm \Psi(\bm X^+)}

\def\p{\psi}

\def\nL{n_L}

\def\Ktilde{\Tilde{\bm K}}
\def\Wtilde{\Tilde{\bm W}}

\newcommand{\ip}[2]{\langle {#1}, {#2} \rangle}
\newcommand{\be}{\begin{equation}}
\newcommand{\ee}{\end{equation}}

\title{\LARGE \bf
Extending dynamic mode decomposition to data from multiple outputs
}

\author{Nibodh Boddupalli$^{1}$
\thanks{$^{1}$Department of Mechanical Engineering and Center for Control, Dynamical Systems, and Computation,
        University of California, Santa Barbara, California, USA.
        {\tt\small nibodh@ucsb.edu}}%
}

\begin{document}

\maketitle
\thispagestyle{empty}
\pagestyle{empty}

\begin{abstract}

System identification based on Koopman operator theory has grown in popularity recently. Spectral properties of the Koopman operator of a system were proven to relate to properties like invariant sets, stability, periodicity, etc. of the underlying system. Estimation of these spectral objects has become widely accessible with the numerous algorithms developed in recent years. We show how one such algorithm -- extended dynamic mode decomposition -- can be used on data from multiple outputs of a system. These outputs that are functions of state are called observables in the literature and could be known outputs like the state itself or unknown outputs like data from sensors used in systems of biological interest. To this end, we approximate the desired observables and their iterates in time using minimizers of regularized least-squares problems which have analytic solutions with heuristic provisions for expected estimates.

\end{abstract}

\section{INTRODUCTION}

Many physical systems exhibit phenomena with unknown governing dynamics. These systems can be high dimensional and often partially observed. Recently, an emerging set of operator-theoretic tools have gained traction, centered on discovering linear representations or approximations of nonlinear dynamical systems in function spaces \cite{Mezic2004, Mezic05}. Originally derived for Hamiltonian systems \cite{Koopman}, popular numerical \cite{Schmid, Rowley09, eDMD, HankelDMD} and theoretical \cite{mezic19} techniques for Koopman operator theory enable input-output perspective \cite{korda2020optimal, MilanMPC, ottokoopmancontrol}, and spectral analysis of nonlinear systems \cite{Mezic05, Korda2018data}.

One of the primary interests in the development \cite{Mezic2004, Mezic05} and further exploration \cite{mezic19, stochastic_example} of Koopman operator framework is its relations to the properties of the underlying dynamical systems as demonstrated in \cite{isostables, korda2020optimal, Korda2018data}, etc. For this, we briefly introduce the Koopman operator. We present the the extended dynamic mode decomposition (EDMD) algorithm \cite{eDMD} as vector and matrix representations of observables and the Koopman operator as projections onto some finite-dimensional function space. EDMD is a modal decomposition algorithm which is a popular numerical method for obtaining finite-sections of the Koopman operator that has proved useful for predictions and control \cite{MilanMPC, edmd_convergence}. Here, we consider either invariance of the EDMD dictionary of functions under action of the Koopman operator or spanning the outputs, and discuss the difference this makes for the estimated spectral objects. We also include output data containing both expected and unknown observables. The flexibility of including expected observables gives the fidelity of including state-variables or other known outputs \cite{johnsonSIL} alongside outputs whose closed form expressions may not be known. To overcome possible under-determined set of equations, we use Tikhonov regularization \cite{tikhonov} which trades sparsity for the advantage of uniquely existing analytic solutions in our use cases.

\section{INTRODUCTION TO KOOPMAN OPERATOR THEORY}

Consider a state $\x \in \M \subseteq \R^n$ whose evolution in time $t \in \R_{\geq 0}$ is given by the non-singular flow of a dynamical system $\St:\M \to \M$ as
\begin{equation*}
    \x(t) = \St(\x(0)),
\end{equation*}
with the state \textit{observed} through output(s) $\y \in \R^p$
\begin{equation*}
    \y(t) = \g(\x(t)).
\end{equation*}
In practice, data from the above system is often available by sampling discretely in time, say, at intervals of $\Delta t$. With interest of data-driven applications, we work with the map $\S \coloneqq \S^{\Delta t}$ as
\be\label{discrete_sys}
\begin{aligned}
    \x^+ &\coloneqq \S(\x)\\
    \y &= \g(\x).
\end{aligned}
\ee

Functions of state like the outputs $\{g_i\}_{i = 1}^p: \M \to \R$ are called ``observables" in Koopman operator literature. While observables like the above outputs are real-valued, including complex-valued observables could ease working with eigenfunctions. If an observable $g$ belongs to a Banach space $\F$ of functions closed under composition $\circ$ with $\S$, the Koopman operator $\K: \F \to \F$ associated with system (\ref{discrete_sys}) is defined as
\begin{equation}\label{discrete_koopman}
    \K g(\x)= g\circ\S(\x) = g(\x^+).
\end{equation}
While we used $\x$ and $\x^+$ in the above, the state is an argument and can be dropped to from the map in function-space which shows how Koopman operator approach gives a \textit{global} perspective on the state-space \cite{lan2013linearization}. Examples of observables include functions of state like a state-variable $g_1(\x) = x_2$, an output $g_2(\x) \eqqcolon y$, possible basis functions like $g_3(\x) = |x_1 - c_i|^2 \log |x_1 - c_i| \eqqcolon \psi_i(\x)$, etc. The book \cite{2020book} offers further references on such properties and more. One of the key properties is linearity of the Koopman operator comes from the linearity of composition as seen from
\begin{multline}\label{linearity}
    \K(\alpha_1 g_1 + \alpha_2 g_2) = (\alpha_1 g_1 + \alpha_2 g_2) \circ \K\\
    =\alpha_1 g_1 \circ \K + \alpha_2 g_2 \circ \K = \alpha_1 (\K g_1) + \alpha_2 (\K g_2),
\end{multline}
with $\alpha_1, \alpha_2 \in \R$. This allows the rich set of tools of analysis and control available from linear systems theory to be exercised in many classes of nonlinear systems \cite{appliedkoopmanism, MilanMPC}.

\subsection{Spectral expansion of observables}

The trade-off for linearity is dimensionality as the Koopman operator is usually infinite-dimensional since its domain is a Banach space of functions. While dimensionality introduces complications, linearity makes it conducive to spectral analysis. In addition to discrete spectral objects like eigenvalues and eigenvectors, infinite-dimensional operators can also have continuous and residual spectrum \cite{banks}. Properties of nonlinear systems were proved as inferable from the spectral properties of its Koopman operator in \cite{Mezic05, mezic19}. Thus, Koopman operator theory extensively uses spectral analysis \cite{Rowley09, 2017_mezic_arbabi, appliedkoopmanism, taira2017modal}, etc. One such finding is that the Koopman operators of systems with hyperbolic\footnote{Not necessarily elliptic \cite{mezic19}} fixed points, periodic orbits, and limit-torii have only discrete spectrum and their associated eigenfunctions span the space of observables \cite{hassan_thesis}. If $\{\phi_j\}_{j = 1}^\infty$ denote the eigenfunctions of such a system, an observable $g \in \F$ can be represented as
\begin{equation*}
g = \sum_{j = 1}^\infty c_j \phi_j,
\end{equation*}
where $\{c_j\}_{j = 1}^\infty$ are called the Koopman modes of $g$. These Koopman modes can be used to decompose the observable into components that evolve linearly in time under $\K: \F \to \F$ using eqn (\ref{discrete_koopman}) as
\begin{equation*}
\K g = \sum_{j = 1}^\infty c_j \K \phi_j = \sum_{j = 1}^\infty c_j \lambda_j \phi_j
\end{equation*}
where $\{\lambda_j\}_{j = 1}^\infty$ denote the corresponding eigenvalues. With multiple observables like in system $(\ref{discrete_sys})$, we can represent the Koopman modes of each $g_i$ as $\{c_{ij}\}_{j = 1}^\infty$. It should be noted that Koopman modes are characteristic of the observable meaning each component of the output $\y$ has its own Koopman modes whereas the eigenfunctions and eigenvalues of the Koopman operator are intrinsic to the underlying system characterised by $\S$. Multiple observables $\{g_i\}_{i = 1}^p$ like that in eqn (\ref{discrete_sys}) can be denoted as a ``vector-observable" $\g$ and their Koopman modes corresponding to the same eigenfunction $\phi_j$ can be denoted as a mode $\bm c_j \coloneqq \begin{bmatrix} c_{1j} & c_{2j} & \cdots & c_{pj}\end{bmatrix}^T$ of the vector-observable \cite{Kic}. Then, the above Koopman mode expansion can be written for $k$ time-steps as
\be\label{KMD}
\K^k \g = \sum_{j = 1}^\infty \bm c_j \lambda_j^k \phi_j
\ee

Since systems of scientific and engineering interests often demonstrate behaviour that is coherent in time, modal decomposition like the above is desired. In practice however, tractable dimensions tend to be finite thus requiring optimal approximation of eigenvalues, eigenfunctions,and Koopman modes. Popular numerical algorithms such as \cite{appliedkoopmanism, eDMD, HankelDMD, williams2015kernel}, etc. have demonstrated their success in estimating finite-dimensional approximations. The literature has mostly worked with known observables but we can easily extend this to using multiple unknown observables like outputs whose closed form may not be known or approximately known.

\section{FINITE DIMENSIONAL APPROXIMATIONS}

Many numerical implementations of Koopman mode decomposition have come as part of the class of Dynamic Mode Decomposition (DMD) algorithms which can themselves be viewed as identification methods on finite-dimensional subspaces. Finding invariant spaces of functions is non-trivial and may not even be practical given finite amounts of data available from finitely samples points in time. Thus, most of these numerical methods make no assumptions on invariance and aim to optimally estimate the spectral objects. For this reason, the most popular algorithms like DMD \cite{Schmid}, EDMD \cite{eDMD, edmd_convergence}, Hankel-DMD \cite{HankelDMD}, Kernel based EDMD \cite{williams2015kernel}, etc. consider inner-product spaces of functions like $l^2$ or  Reproducing Kernel Hilbert Spaces (RKHS) \cite{williams2015kernel, rosenfeld2019occupationkernels, rosenfeld2019occupation, rosenfeld2021control, klus2016numerical} on the sampled data. Our work is based on EDMD \cite{eDMD} which is also the used for many proofs \cite{edmd_convergence}, extensions \cite{klus2016numerical, deepDMD}, modifications \cite{williams2015kernel, li2017extended, rosenfeld2019occupationkernels}, control \cite{MilanMPC, rosenfeld2021control}, etc. and has been derived as a linear regression problem in \cite{thesis} based on what was shown in \cite{finitesection}.

\subsection{Analytic EDMD without invariant dictionary}
We use a finite number of observables $\{\p_i\}_{i = 1}^{\nL}: \M \to \mathds{R}$ in a Hilbert space $\F$. The space spanned by these observables $\Ftilde \coloneqq span\{\p_1,\p_2,\dots,\p_{\nL}\}$ is a $\nL$-dimensional subspace of the Hilbert space $\Ftilde \subset \F$. Since $\Ftilde$ is a closed linear subspace of $\F$, we know from the projection theorem \cite{luenberger1997optimization} the existence of a unique minimizer $f \in \Ftilde$
\be\label{projection_definition}
\P(f) = \argmin_{\Tilde{f} \in \Ftilde}\|\Tilde{f} - f \|,
\ee
which is  the orthogonal projection of $f$ onto $\Ftilde$ using the projection $\P: \F \to \Ftilde$. Here $\|\cdot\|$ is the inner-product over $\M$. Usually, observables $\g$ are all known and in the span of a chosen set of dictionary functions. If $\{g_i\}_{i = 1}^p \in \Ftilde$ and denoting $\begin{bmatrix} \p_1 & \p_2 & \cdots & \p_{\nL} \end{bmatrix} \eqqcolon \bm \p^T$,
\be\label{g_exact}
\begin{aligned}
    \bm g^T &= \begin{bmatrix} g_1 & \cdots & g_p \end{bmatrix} = \begin{bmatrix}
\bm \p^T \bm w_1 & \cdots & \bm \p^T \bm w_p \end{bmatrix}\\
&= \bm \p^T \left[
        \begin{array}{cccc}
        \vrule &    & \vrule \\
        \bm w_1 &  \cdots & \bm w_p \\
         \vrule &    & \vrule \\
        \end{array}
        \right] \eqqcolon \bm \p^T \bm W.
        \end{aligned}
\ee
If it is \textit{not necessary} that $\{g_i \circ \S\}_{i = 1}^p \in \Ftilde$ i.e. $g_i \in \Ftilde \nRightarrow \K g_i \in \Ftilde$, we can use the above with eqns (\ref{projection_definition}, \ref{discrete_koopman}, \ref{linearity}) to write
\be\label{kg_proj}
\begin{aligned}
\P(\g^T \circ \S) &= \P(\K \bm \p^T \bm W)\\&= \begin{bmatrix} \P(\K \psi_1) &  \cdots & \P(\K \psi_{\nL})\end{bmatrix} \bm W\\
&= \bm \p^T \left[
        \begin{array}{cccc}
        \vrule &    & \vrule \\
        \Tilde{\bm k}_1 &  \cdots & \Tilde{\bm k}_{\nL} \\
         \vrule &    & \vrule \\
        \end{array}
        \right] \bm W \eqqcolon \bm \p^T \Ktilde \bm W.
\end{aligned}
\ee
Here, the matrix approximation of the Koopman operator $\Ktilde = \P\K$ is obtained analytically \cite{edmd_convergence} as
\be\label{Galerkin}
\Ktilde = \bm G^{\dagger} \bm A
\ee
where
$$\bm G = \begin{bmatrix} \ip{\p_1}{\p_1} &  \cdots & \ip{\p_1}{\p_{\nL}}\\
\vdots &  \ddots & \vdots\\
\ip{\p_{\nL}}{\p_1} &  \cdots & \ip{\p_{\nL}}{\p_{\nL}} \end{bmatrix}
$$
and
$$
\bm A = \begin{bmatrix}
 \ip{\p_1}{\p_1 \circ \S} & \cdots & \ip{\p_1}{\p_{\nL} \circ \S} \\ \vdots & \ddots & \vdots \\  \ip{\p_{\nL}}{\p_1 \circ \S} & \cdots & \ip{\p_{\nL}}{\p_{\nL} \circ \S}
 \end{bmatrix}
$$
with $\ip{\cdot}{\cdot}$ denoting the inner-product in $\F$ over $\M$ and $(\cdot)^\dagger$ denoting the Moore-Penrose pseudoinverse.

From the above, the eigenfunctions $\{\phi_i\}_{i = 1}^{\nL} \in \Ftilde$ of $\K$ are obtained using the eigenvectors $\{\bm v_i\}_{i = 1}^{\nL}$ of $\Ktilde$ as
\be\label{eigendecomposition}
\phi_i = \bm \p^T \bm v_i
\ee
The above estimation of eigenfunctions and their corresponding eigenvalues $\{\lambda\}_{i = 1}^{\nL}$ was proved to converge depending on the computed inner-products in $\cite{edmd_convergence}$. If we denote all the eigenvectors as the matrix
$$\bm V \coloneqq \begin{bmatrix} \bm v_1 & \bm v_2 & \cdots & \bm v_{\nL}\end{bmatrix},$$
the Koopman modes $\{\bm c_i\}_{i = 1}^{\nL}$ of $\g$ in eqn (\ref{KMD}) are obtained \cite{eDMD} as
\be\label{Koopman_modes}
\begin{bmatrix} \bm c_1 & \cdots & \bm c_{\nL} \end{bmatrix} = (\bm V^{-1} \bm W)^T.
\ee

\subsection{Unknown observables with invariant dictionary}

When the observables are unknown, we project them onto $\Ftilde$ and approximate them as done when $\K \psi_i \notin \Ftilde$ similar to eqn (\ref{Galerkin}) as
\be\label{projected_observable}
\P(\bm g)^T = \bm \psi^T \left[
        \begin{array}{cccc}
        \vrule &    & \vrule \\
        \Tilde{\bm w}_1 &  \cdots & \Tilde{\bm w}_p \\
         \vrule &    & \vrule \\
        \end{array}
        \right] \eqqcolon \bm \psi^T \Wtilde,
\ee
where $\Wtilde$ obtained as
\be\label{Estimating_coeff}
\Wtilde = \bm G^\dagger \begin{bmatrix}
 \ip{\psi_1}{g_1} & \cdots & \ip{\psi_1}{g_p} \\ \vdots & \ddots & \vdots \\  \ip{\psi_{\nL}}{g_1} & \cdots & \ip{\psi_{\nL}}{g_p}
 \end{bmatrix}
 \eqqcolon \bm G^\dagger \bm B,
\ee
uniquely minimizes $\|\bm \psi^T \Tilde{\bm w}_i - g_i\| \ \forall \ i = 1, 2, \cdots, p$ analogous to equation (\ref{kg_proj}).

\subsection{Known outputs versus invariant dictionary}

To elucidate the challenge in estimating the Koopman-tuple $\{(\phi_i, \lambda_i, \bm c_i)\}_{i = 1}^{\nL}$ from observables like multiple outputs of which some observables like state variables may be known \cite{johnsonSIL} -- using a dictionary of functions not necessarily invariant to the action of the Koopman operator -- we start with decomposing the function space on which the Koopman operator itself is defined. Working in a Hilbert space of functions $\F$ allows us to decompose it into the subspace $\Ftilde$ determined by our choice of dictionary functions and its unique orthogonal complement \cite{luenberger1997optimization} $\Ftilde^c$ as 
\[\F = \Ftilde \oplus \Ftilde^c.
\]

Given some unknown observable $g \in \F$, we can use eqn (\ref{projection_definition}) to decompose it into its projection and skew-projection as
\[
g = \P g + (\I - \P)g.
\]
The action of the Koopman operator on the above can be written as
\[\K g = \K\P g + \K(\I - \P)g.
\]
Since $\K g \in \F$, we can use the above to further decompose $\K$ using linearity as
\begin{multline*}
    \K = \underbrace{\P \K}_{\Ftilde \to \Ftilde} \P + \underbrace{(\I - \P) \K}_{\Ftilde \to \Ftilde^c} \P\\
+ \underbrace{\P \K}_{\Ftilde^c \to \Ftilde} (\I - \P) + \underbrace{(\I - \P) \K}_{\Ftilde^c \to \Ftilde^c} (\I - \P).
\end{multline*}
By identifying the domain and co-domain of each of the operators in the above, we can denote this as
\be\label{operatar_decomp}
\begin{aligned}
\K = \left[\begin{array}{cc} \P \K_{\Ftilde} & \P \K_{\Ftilde^c} \\ (\I - \P) \K_{\Ftilde} & (\I - \P) \K_{\Ftilde^c} \end{array}\right] : \begin{array}{c} \Ftilde \\ \oplus \\ \Ftilde^c \end{array} \to \begin{array}{c} \Ftilde \\ \oplus \\ \Ftilde^c \end{array}.
\end{aligned}
\ee

In the above, we see that when $\Ftilde$ is chosen to be a space spanned by a subset of eigenfunctions that also span desired observables as done in \cite{korda2020optimal, folkestad2020extended, kaiser2017data, haseli2021learning, Korda2018data}, etc. it has a block-diagonal decomposition. With only one of these assumptions, either invariance $\Ftilde$ to the action of the Koopman operator or observable $g \in \Ftilde$ give upper or lower block-triangular decomposition respectively. Without either assumption, we do not have such a block-triangular structure that possibly increases ``uncertainty" in matrix approximation as no vector in $\Ftilde^c$ can be estimated. To show this, we use the following result

\begin{lemma}\label{proj_lemma}
For $g \in \F$,$$\|g\| \geq \|\P g\|$$with $\|g\| = \|\P g\|$ only if $g = \P g$ 
\end{lemma}
\begin{proof}
Using eqn (\ref{projection_definition})
$$g = \P g + (\I-\P) g$$
\begin{multline*}
    \ip{g}{g} = \ip{\P g }{\P g } + \ip{(\I-\P) g}{(\I-\P) g}\\+ 2\cancelto{0}{\ip{\P g }{(\I-\P) g}}
\end{multline*}
which gives
$$\|g\|^2 - \|\P g\|^2 = \|(\I - \P) g\|^2.$$
Since $\|\cdot\| \geq 0$,$$\|g\| \geq \|\P g\|.$$
\end{proof}

First we show how non-zero residual under action of the Koopman operator shows that the dictionary is not invariant to the Koopman operator if we know that our chosen dictionary spans our desired observable.

\begin{claim}
If $g \in \Ftilde$, then$$\|\bm \psi^T \bm w^+\| \leq \|\bm \psi^T \Ktilde \bm w\|$$
with $\|\bm \psi^T \bm w^+\| < \|\bm \psi^T \Ktilde \bm w\| \implies \|(\I - \P)\big(\K \P g\big)\| > 0$
\end{claim}
\begin{proof}
If $g \in \Ftilde$, then $\|\P\big(\K (\I - \P)g \big)\| = 0$. Using this in eqn (\ref{operatar_decomp}),
$$\P \K g = \P \K \P g + \cancelto{0}{\P \K (\I - \P) g}.$$
Then, from Lemma \ref{proj_lemma} and eqns (\ref{g_exact}, \ref{kg_proj})
$$\|\bm \psi^T \bm w^+\| \leq \|\bm \psi^T \Ktilde \bm w\|$$
with 
$$\|\bm \psi^T \bm w^+\| < \|\bm \psi^T \Ktilde \bm w\| \implies \|(\I - \P)\big(\K \P g\big)\| > 0.$$
\end{proof}

This shows that our dictionary needs to be expanded to reduce residuals and better approximate the Koopman operator and especially its spectrum as strong convergence does not guarantee spectral convergence \cite{edmd_convergence}. We can also infer using eqn (\ref{Koopman_modes}) that although the desired observable $g \in \Ftilde \implies g = \bm \p^T \bm w$, the Koopman modes are not necessarily exact if estimated eigenfunctions are only approximations. We contrast this with when our chosen dictionary does not span the desired observable.

\begin{claim}
If $\Ftilde$ is invariant to the action of the Koopman operator but the residual
\[\|\bm \p^T \Tilde{\bm w}^+ - \bm \psi^T \Ktilde \tilde{\bm w}\| > 0,
\]
is non-zero, then$$g \notin \Ftilde.$$
\end{claim}
\begin{proof}
From the assumption of invariance of the Koopman operator to $\Ftilde$,
\[\K \P g = \P \K \P g + \cancelto{0}{(\I - \P)\K \P g}.
\]Then, from Lemma \ref{proj_lemma} and eqn (\ref{projected_observable}) and eqn (\ref{Galerkin}) which becomes an exact representation of the Koopman operator in $\Ftilde$
$$\|\bm \psi^T \tilde{\bm w}^+\| \geq \|\bm \psi^T \bm K \tilde{\bm w}\|$$
with 
$$\|\bm \psi^T \tilde{\bm w}^+\| > \|\bm \psi^T \bm K \tilde{\bm w}\| \implies \|\P\big(\K (\I - \P)g \big)\| > 0.$$
Then, $$\|(\I - \P)g\| > 0.$$
\end{proof}

From invariance, we remark that the estimated eigenfunctions are exact \cite{finitesection}. The non-zero residual $\|g - \bm\p^T \tilde{\bm w}\|$ shows that not all Koopman modes of $g$ are captured but we know that those modes corresponding to the captured eigenfunctions are accurate and will evolve linearly due to the invariance \cite{eDMD}. This also shows why close approximations of the eigenfunctions are desired as they offer invariant subspaces by their very definition. Algorithms for better approximation of spectral objects thus remains an active area of research \cite{korda2020optimal, kaiser2017data, mezic19, Korda2018data}.

\section{NUMERICAL ESTIMATION}

Numerical estimation of the Koopman matrix is not far from the analytic expressions seen  earlier, only here we work with finite dimensional spaces. From equation (\ref{Galerkin}), accuracy of the numerical estimations $\Ktilde$ would depend on our approximations of the inner products that are entries of $\bm G$ and $\bm A$. Ideally, these would be calculated over the entire set of states that the dynamical system takes which is $\M$. If we want to consider only a subset $\M_k \subset \M$, we can use a measure $\mu_k$ supported \textit{only} on $\M_k$. However finite amounts of discretely sampled data changes the sets over which the inner-products are estimated \cite{edmd_convergence}. Consider the following set of states in $\M$
\[\X \coloneqq \{\x_1, \x_2, \cdots , \x_m\}.
\]
An empirical measure $\mu_{\X}$ can be defined over these points using Dirac measure $\delta_{\x_i}$ centered at each $\x_i$ to approximate \cite{eDMD}
\be\label{matrix_entries}
\begin{aligned}
G_{ij} = \sum_{k=1}^m\frac{\psi_i^*(\x_k) \psi_j(\x_k)}{m} &\implies \bm G = \frac{\data^* \data}{m}\\
A_{ij} = \sum_{k=1}^m\frac{\psi_i^*(\x_k) \psi_j(\S(\x_k))}{m} &\implies \bm A = \frac{\data^*\bm\Psi(\bm X^+)}{m}
\end{aligned}
\ee
under appropriate sampling \cite{edmd_convergence, thesis} where
\[
\begin{aligned}
\X^+ &\coloneqq \S(\X) = \{\S(\x_1), \S(\x_2), \cdots , \S (\x_m)\}\\
\data &\coloneqq \begin{bmatrix}
\bm \psi(\x_1)^T \\  \vdots \\ \bm \psi(\x_m)^T 
\end{bmatrix} \quad
\dataplus \coloneqq \begin{bmatrix}
\bm \psi(\S(\x_1))^T \\ \vdots \\ \bm \psi(\S(\x_m))^T 
\end{bmatrix}
\end{aligned}
\]
and $\data^*$ denotes complex-conjugate transpose of $\data$. Using the above in eqn (\ref{Galerkin}), the numerically estimated $\Ktilde$ is
\[
\Ktilde = \data^\dagger \dataplus,
\]
This is the form commonly presented as the interpretation of the Koopman matrix being the best approximation of the flow as a linear map since this minimizes the residual
\be\label{koopman_matrix_residual}
\begin{aligned}
&\min\|\K \psi_i - \bm \psi^T \tilde{\bm k}_i\|_{L^2(\mu_{\X})}\text{ for each } i = 1, 2, \cdots, \nL\\
&\implies \min\|\dataplus - \data \Ktilde\|_F.
\end{aligned}
\ee
This gives an exact representation (i.e. 0 residual in $\|\cdot\|_{L^2(\mu_{\X})} \ \forall \ i = 1, \cdots, \nL$) of the Koopman operator if our chosen dictionary $\{\psi_i\}_{i=1}^{\nL}$ spans $\{\K \psi_i\}_{i=1}^{\nL}$ over $\X$. When the dictionary functions are linearly independent, $\bm G$ is invertible and $\bm G^\dagger = \bm G^{-1}$. When under-determined like in the case of linearly dependent dictionary functions or insufficient data, a continuous family of minimizers exists and the MP-pseudoinverse provides the one with the minimum norm via (truncated) Singular Value Decomposition (SVD) \cite{luenberger1997optimization}. When equation (\ref{koopman_matrix_residual}) is solved under constraints on entries of $\Ktilde$, it results in regularization. For example, $L^2$-regularization \cite{ridge} can be used to analytically solve
\be\label{l2_regularization}
\begin{aligned}
&\min\|\K \bm \psi_i - \bm \psi^T \tilde{\bm k}_i\|_{L^2(\mu_{\X})} + \beta \|\tilde{\bm k}_i\|_2\text{ for  }i = 1, \cdots, \nL\\
&\implies \min\|\dataplus - \data \Ktilde\|_F + \beta \|\Ktilde\|_F,
\end{aligned}
\ee
using
\[\Ktilde = (\bm G + \beta \bm I)^{-1}\bm A,
\]
where $\beta \geq 0$ is a "regularizer" that relates to the desired optimization constraint. Iterations of the above approximation followed by gradient descent towards optimal dictionary functions constitutes deep learning based EDMD \cite{li2017extended}.

\subsection{Numerical estimation of unknown outputs}

In Koopman operator literature, observables are functions of state within the domain of the Koopman operator. Outputs are also observables by definition \cite{2020book, MilanMPC, korda2020optimal}. We denote the output data as
\[\bm Y \coloneqq \begin{bmatrix}
\bm y_1 & \bm y_2 & \cdots & \bm y_m
\end{bmatrix} \in \mathds{R}^{p \times m}
\]
where $\bm y_i = \bm g(\x_i)$ since the closed form of output $\bm g(\x)$ is not known. We repeat the procedure for estimating $\Ktilde$. we refer to equation (\ref{Estimating_coeff}) and use the above to estimate $\bm B$ like that in equation (\ref{matrix_entries})
\[\bm B = \frac{1}{m} \data^*\bm Y^T,
\]
and get
\be\label{edmd_C}
\Wtilde = \bm G^\dagger\bm B.
\ee

\subsection{Regularization and expected outputs}

Regularization can be be used here also to minimize the above residual subject to constraints from apriori experience. As outputs are often determined by sensors from high dimensional systems like velocities at some spatial position \cite{Schmid, Rowley09, HankelDMD} or even from control systems \cite{MilanMPC, korda2020optimal}, we can guess some possible forms like $\bm y = \bm C \x$ which can be included in the dictionary alongside other possible basis like the inclusion of state itself in \cite{johnsonSIL, deepDMD}. Consider a dictionary
\[\bm \psi^T = \begin{bmatrix} (\bm C\x)^T & \bm z_d(\sigma(\bm z_{d-1}(\sigma( \cdots (\sigma (\bm z_1))\dots))^T\end{bmatrix}
\]
where $\bm z_d = \bm H_d \bm \x + \bm r_d \in \mathds{R}^{\nL - p}$ with $d$ being the depth of a neural network with activation functions $\sigma$ that are known to demonstrate universal function approximation property (UFAP) in practice like rectified linear units \cite{UFAP_relu}, sigmoidal functions \cite{UFAP_sigmoid}, radial basis functions (RBFs) \cite{UFAP_RBF}, etc. However, universal function approximators in the dictionary could approximately/exactly span our desired functions that are also in the dictionary \cite{thesis} causing ill-conditioned/degenerate $\bm G$.

In this case, we can use Tikhonov regularization \cite{tikhonov} which is broader than the ridge regression used in equation (\ref{l2_regularization}) to take expected vector representations of outputs into account. We expect $\tilde{w}_i^T \approx \begin{bmatrix} \bm c_i^T & \mathbb{0}^T \end{bmatrix}$ but the minimum norm solution using equation (\ref{edmd_C}) could be different. Consider expected vector representations of first $l$ outputs as $\bm w^0_1, \bm w^0_2, \cdots, \bm w^0_l, 1 \leq l \leq p$ and further constraining the sum of their squares using $\|\tilde{\bm w}_i\|_{\bm Q} = \tilde{\bm w}_i^T \bm Q \tilde{\bm w}_i$ while we lack apriori information on the rest $p - l$ outputs
\be\label{tikhonov_regularization}
\begin{aligned}
&\min\|g_i - \bm \psi^T \tilde{\bm w}_i\|_{L^2(\mu_{\X})} + \begin{cases*} \|\tilde{\bm w}_i - \bm w^0_i\|_{\bm Q}  \text{ } i = 1, \cdots, l\\
\|\tilde{\bm w}_i\|_{\bm Q}   \text{ } i = l+1, \cdots, p \end{cases*}\\
&\implies \min\|\bm Y^T - \data \Wtilde\|_F + \|\Wtilde - \Wtilde_0\|_{\bm Q}
\end{aligned}
\ee
where
\[\bm W^0 = \left[
        \begin{array}{cccccc}
        \vrule &    & \vrule & \vrule & & \vrule\\
        \bm w^0_1 &  \cdots & \bm w^0_l & \mathbb{0} & \cdots & \mathbb{0}\\
         \vrule &    & \vrule & \vrule & & \vrule\\
        \end{array}
        \right], \quad \bm Q \succ 0,
\]
is solved by
\[\Wtilde = (\bm G + \bm Q)^{-1}(\bm B + \bm Q \bm W^0).
\]
This can be used in an example problem. If the full state $\bm y = \bm x$ is considered alongside unknown outputs like that done in \cite{twat}, we can use $\bm C = \bm I_{n \times n}$ in our example dictionary and weight all dictionary functions equally using $\bm Q_1 = \beta_1 \bm I_{\nL \times \nL}$ for states $i = 1, \cdots, n$ and $\bm Q_2 = \beta_2 \bm I_{\nL \times \nL}$ for outputs. Then the expected output matrix using block matrix notation is
\[\bm W^0 = \begin{bmatrix} \bm I_{n \times n} & 0 \\ 0 & 0 \end{bmatrix} \eqqcolon \begin{bmatrix} \bm W^0_1 & 0\end{bmatrix}
\]
which gives the representations $\Wtilde_1 = (\bm G + \beta_1 \bm I)^{-1}(\bm B_1 + \bm \beta_1 \bm W^0_1)$ for the states and $\Wtilde_2 = (\bm G + \beta_2 \bm I)^{-1}\bm B_2$ for the outputs. Here, $\beta_1 > \beta_2$ would penalize $\|\Wtilde_1 - \bm W^0_1\|_F$ stronger than $\|\Wtilde_2\|_F$ giving representations of state closer to the expected vectors while allowing looser bounds on the vector representations of the remaining outputs as they are not known. While these parameters could be heuristic in practice, they could be considered analogous to hyper-parameters in machine learning literature.

\section{CONCLUSION}

We considered the Koopman operator estimation problem for data available from multiple outputs. Estimation of spectral objects of the Koopman operator remains an active area of research with novel numerical algorithms being developed in pursuit of accuracy. On the application side of these algorithms, there seems to be a disconnect between what are considered observables. Here we showed how output data from sensors in systems of interest can be considered in existing numerical implementations of Koopman operator theory. We recapitulated how eigenfunctions and eigenvalues of the Koopman operator are characteristic of the underlying system while Koopman modes are dependent on the observables in question. We contrasted the choice of dictionary of functions that span the outputs but not invariant to the Koopman operator against that of a dictionary of functions that is invariant but does not span the outputs. We also derived closed form expressions to numerically estimate the involved functions using Tikhonov regularization to gives unique solutions that minimize residuals in an inner-product sense. We also considered the case of incorporating data from expected outputs like state-variables into this alongside data from unknown outputs like sensors in experiments.

\section{ACKNOWLEDGEMENT}

We thank Alexander ``Sasha" Davydov for insightful discussions on this topic and valuable feedback on this work. We would also like to thank Enoch Yeung for overseeing precursory work which was graciously supported by a fellowship from the Center for Control, Dynamical Systems, and Computation (CCDC) at the University of California Santa Barbara.

\bibliographystyle{IEEEtran}
\bibliography{IEEEabrv,root}

\end{document}